\documentclass[a4wide,12pt,reqno]{article}
\usepackage{bm}
\usepackage{mathrsfs}
\usepackage[nottoc]{tocbibind} 
\usepackage{fullpage}
\usepackage{enumerate}
\usepackage[table]{xcolor}
\usepackage[utf8]{inputenc}
\usepackage{amssymb, amsmath, amsfonts}
\usepackage{mathrsfs}
\usepackage{epsfig,amsbsy,amsthm} 
\usepackage{tikz}
\usetikzlibrary{patterns,calc,decorations.markings}
\usepackage{graphics}
\usepackage{caption}
\usepackage{subcaption}
\usepackage{float, epsfig}
\usepackage{fixmath}
\usepackage{graphics}
\usepackage{pgfpages}
\usepackage{appendix}
\renewcommand{\arraystretch}{3}
\numberwithin{equation}{section}  
\usepackage{graphicx}
\usepackage{pst-all}
\usepackage{calligra}
\usepackage{tikz}
\usepackage{calligra}
\usepackage{color}
\usepackage{tikz}
\usepackage{cleveref}
\usepackage{algorithm}
\usepackage{multirow}
\usepackage{array}
\newcolumntype{H}{>{\setbox0=\hbox\bgroup}c<{\egroup}@{}}
\usepackage{accents}
\DeclareMathAlphabet{\mathpzc}{OT1}{pzc}{m}{it}
\DeclareMathAlphabet{\mathcalligra}{T1}{calligra}{m}{n}
\DeclareMathAlphabet{\mathpzc}{OT1}{pzc}{m}{it}
\DeclareMathAlphabet{\mathcalligra}{T1}{calligra}{m}{n}
\DeclareMathAlphabet{\mathpzc}{OT1}{pzc}{m}{it}
\DeclareMathAlphabet{\mathcalligra}{T1}{calligra}{m}{n}
\begin{document}
\newtheorem{theorem}{\bf Theorem}[section]
\newtheorem{proposition}[theorem]{\bf Proposition}
\newtheorem{definition}{\bf Definition}[section]
\newtheorem{corollary}[theorem]{\bf Corollary}
\newtheorem{exam}[theorem]{\bf Example}
\newtheorem{remark}[theorem]{\bf Remark}
\newtheorem{lemma}[theorem]{\bf Lemma}
\newtheorem{assum}[theorem]{\bf Assumption}
\newcommand{\von}{\vskip 1ex}
\newcommand{\vone}{\vskip 2ex}
\newcommand{\vtwo}{\vskip 4ex}
\newcommand{\ds}{\displaystyle}
\def \noin{\noindent}
\newcommand{\be}{\begin{equation}}
\newcommand{\eeno}{\end{equation*}}
\newcommand{\ba}{\begin{align}}
\newcommand{\ea}{\end{align}}
\newcommand{\bano}{\begin{align*}}
\newcommand{\eano}{\end{align*}}
\newcommand{\bea}{\begin{eqnarray}}
\newcommand{\eea}{\end{eqnarray}}
\newcommand{\beano}{\begin{eqnarray*}}
\newcommand{\eeano}{\end{eqnarray*}}
\newcommand{\chatv}{\hatv{c}}
\def \noin{\noindent}
\def\arraystretch{1.3}
\def \tcK{{\tilde {\mathcal K}}}    
\def \O{{\Omega}}
\def \cT{{\mathcal T}}
\def \cV{{\mathcal V}}
\def \cE{{\mathcal E}}
\def \R{{\mathbb R}}
\def \V{{\mathbb V}}
\def \S{{\mathbb S}}
\def \N{{\mathbb N}}
\def \Z{{\mathbb Z}}
\def \Mc{{\mathcal M}}
\def \Cc{{\mathcal C}}
\def \Rc{{\mathcal R}}
\def \Ec{{\mathcal E}}
\def \Gc{{\mathcal G}}
\def \Tc{{\mathcal T}}
\def \Qc{{\mathcal Q}}
\def \Ic{{\mathcal I}}
\def \Pc{{\mathcal P}}
\def \Oc{{\mathcal O}}
\def \Uc{{\mathcal U}}
\def \Yc{{\mathcal Y}}
\def \Ac{{\mathcal A}}
\def \Bc{{\mathcal B}}
\def \k{\mathpzc{k}}
\def \Rp{\mathpzc{R}}
\def \Jc{\mathpzc{J}}
\def \Os{\mathscr{O}}
\def \Js{\mathscr{J}}
\def \Es{\mathscr{E}}
\def \Fs{\mathscr{F}}
\def \Qs{\mathscr{Q}}
\def \Ss{\mathscr{S}}
\def \Cs{\mathscr{C}}
\def \Ds{\mathscr{D}}
\def \Ms{\mathscr{M}}
\def \Ts{\mathscr{T}}
\def \LL{L^{\infty}(L^{2}(\Omega))}
\def \LH{L^{2}(0,T;H^{1}(\Omega))}
\def \B {\mathrm{BDF}}
\def \el {\mathrm{el}}
\def \re {\mathrm{re}}
\def \e {\mathrm{e}}
\def \div {\mathrm{div}}
\def \CN {\mathrm{CN}}
\def \Rs   {\mathbf{R}_{{\mathrm es}}}
\def \Rb {\mathbf{R}}
\def \Jb {\mathbf{J}}
\def  \apos {\emph{a posteriori~}}
\def\mean#1{\left\{\hskip -5pt\left\{#1\right\}\hskip -5pt\right\}}
\def\jump#1{\left[\hskip -3.5pt\left[#1\right]\hskip -3.5pt\right]}
\def\smean#1{\{\hskip -3pt\{#1\}\hskip -3pt\}}
\def\bsmean#1{\Big\{\hskip -4.5pt\Big\{#1\Big\}\hskip -4.5pt\Big\}}
\def\sjump#1{[\hskip -1.5pt[#1]\hskip -1.5pt]}
\def\jumptwo{\jump{\frac{\p^2 u_h}{\p n^2}}}
\def  \apos {\emph{a posteriori~}}
\def\mean#1{\left\{\hskip -5pt\left\{#1\right\}\hskip -5pt\right\}}
\def\jump#1{\left[\hskip -3.5pt\left[#1\right]\hskip -3.5pt\right]}
\def\smean#1{\{\hskip -3pt\{#1\}\hskip -3pt\}}
\def\bsmean#1{\Big\{\hskip -4.5pt\Big\{#1\Big\}\hskip -4.5pt\Big\}}
\def\sjump#1{[\hskip -1.5pt[#1]\hskip -1.5pt]}
\def\jumptwo{\jump{\frac{\p^2 u_h}{\p n^2}}}
\title{Optimal control of fractional Poisson equation from non-local to local}

\author{Kedarnath Buda \thanks{Department of Mathematics \& Statistics, Indian Institute of Technology Kanpur, Kanpur - 208016,
India; \tt{( Kedarnath.buda@gmail.com)}} 
~~ B. V. Rathish Kumar\thanks{Department of Mathematics \& Statistics, Indian Institute of Technology Kanpur, Kanpur - 208016,
India; \tt{( drbvrk11@gmail.com)}.} ~~ and ~~ Ram Manohar \thanks{Department of Mathematics \& Statistics, Indian Institute of Technology Kanpur, Kanpur - 208016,
India; \tt{(rmanohar@gmail.com)}.}}

\date{}

\maketitle

\textbf{Abstract.}{\small{
In this article, the limiting behavior of the solution $\bar u_s$ of the optimal control problem subjected to the fractional Poisson equation 
$$(-\Delta)^s u_s(x)=f_s(x), \quad x\in \Omega$$
defined on domain $\Omega$ bounded by smooth boundary with zero exterior boundary conditions $u_s(x)\equiv 0, \quad x \in \Omega^c $ is established. We will prove that  $\lim_{s\to 1^-} \bar u_s= \bar u$,  where $\bar u$ is a solution of the optimal control problem subjected to classical Poisson equation  $-\Delta u(x)=f(x), \quad x \in \Omega$ and  $u(x)=0, \quad  x\in \partial \Omega.$
}}  \\ \smallskip
	
\textbf{Key words.}
Fractional Laplacian; $\Gamma$-convergence;  Weak solution; optimal control; State constraints

\section{Introduction}
\label{Sec:1}
 There have recently been significant developments in nonlocal PDE problems owing to their applications in various fields of science and engineering. Today, nonlocal PDE problems find their applications in the areas of turbulence, elasticity, image processing, finance model, geo-magnetic, transport, and diffusion (see \cite{bakunin2008turbulence}, \cite{dipierro2015dislocation}, \cite{gilboa2009nonlocal}, \cite{meerschaert2012fractional}). There have been intensive theoretical studies in nonlocal PDE problems involving Fractional Laplacian (see \cite{di2012hitchhikers}, \cite{biccari2017poisson} \cite{cozzi2017interior}, \cite{ros2014dirichlet}) in the last few decades. The fractional Laplacian is considered as a nonlocal operator because its value at any point depends on the function's values at all other points in the domain. Unlike the classical Laplacian, which involves only local derivative, the fractional laplacian is defined through an integral over the entire space. at any point $x$, the value of $(-\Delta)^su(x)$ depends on the value of $u(y)$ for all $y\in \mathbb{R}^N$. This is why a careful, rigorous, and intensive study must be done to understand it. Through comprehensive developments, fractional Laplacian may be defined in fractional Sobolev spaces in multiple ways. Mathematicians from both the the pure and applied mathematics domain have thoroughly worked on the PDEs involving the fractional Laplacian. In image processing, denoising is an essential breakthrough for image construction. The Total Variation (TV) regularization is being used to denoise the image (see \cite{rudin1992nonlinear}). The total variation minimizing function (classical ROF-denoising approach) is defined as: 

$$\min_{u\in L^2(\Omega)}J(u)= \frac{1}{2}\int_{\Omega}|\nabla u|^2+ \frac{\mu}{2}\|u-f\|^2_{L^2(\Omega)}.$$

Later in \cite{antil2017spectral}, the author replaces the square of semi-norm $\int_{\Omega}|\nabla u|^2 $  with square of fractional Gagliardo-Nirenberg semi-norm, $\int_\Omega |(-\Delta)^{\frac{s}{2}}u|^2$  of fractional order $s\in (0,1)$ and modified fractional variational problem is 
$$\min_{u}\frac{1}{2} \int_\Omega |(-\Delta)^{\frac{s}{2}}u|^2 + \frac{\mu}{2}\|u-f\|^2_{L^2(\Omega)}.$$
where $u$ be the denoised version of $f \in L^2(\Omega)$, the noisy image, $\mu$ is the regularization parameter and $\Omega$ is the domain of image. 
\section{Preliminaries: Analytical Setting}
In this section, we set the basic notation and terminology which is required to prove our result. Firstly we will understand the Fractional laplacian operator and the space where it is defined. Many recent advances have also been discussed, which hold a great deal of significance in studying the Fractional Laplacian and its optimization. Let us define a fractional Laplacian operator in $\mathbb{R}^N$, which is a nonlocal operator defined by the singularity of the integral at $x\in \mathbb{R}^N$, 
\begin{equation}\label{eq:2.1}
    (-\Delta)^su(x) =C_{N,s} \text{P.V.} \int_{\mathbb{R}^N} \frac{u(x)-u(y)}{|x-y|^{N+2s}} \; \mathrm{d} y 
 \end{equation}
where 
$$C_{N,s}:= \frac{s2^{2s}\Gamma(\frac{N+2s}{2})}{\pi^{N/2}\Gamma(1-s)}. \label{e1}$$ 

The integration defined above is in Principal value sense, i.e., the singular integral as for $x\in \mathbb{R}^N$,
$$(-\Delta)^su(x) =C_{N,s}\lim_{\epsilon \to 0} \int_{\mathbb{R}^N \setminus B_\epsilon(x)} \frac{u(x)-u(y)}{|x-y|^{N+2s}}\;\mathrm{d}y \label{e11}.$$

To make sense, the integral \ref{eq:2.1}, we have to take  $u\in \mathcal{C}^{1,1}_{loc} \cap \mathcal{L}_s^1$,
where 
$$\mathcal{L}_s^1(\mathbb{R}^N)= \left\{u\in \mathcal{L}_{loc}^1| \int_{\mathbb{R}^N} \frac{|u(x)|}{1+|x|^{N+2s}}\;\mathrm{d}x < \infty \right\}.$$

We can easily verify that for the integration \ref{eq:2.1} defined above, the P.V. sense in it can be ignored for $s\in(0,\frac{1}{2})$ (see \cite{di2012hitchhikers}) because the integral will no more be singular near $x$ for such $s$.\\

If a regular function $u$ on $\R^N$ is identically zero near the point $x$, then $\Delta u(x)=0$. But in the case of fractional, this is not true, as in our bounded domain $\Omega$, let us consider $u>0 $ in $\Omega $ and $u\equiv 0$ in $\mathbb{R}^N\setminus \Omega$.
For $x\in \mathbb{R}^N\setminus \Omega $, we have, 

$$(-\Delta)^su(x) =C_{N,s} \text{P.V.} \int_{\Omega} \frac{-u(y)}{|x-y|^{N+2s}} \;\mathrm{d}y < 0.$$
Here, we can notice that the integral is negative. Hence, its integration is negative over $\Omega$ at each $x \in \mathbb{R}^N$. The value of $(-\Delta)^{s}u(x)$ depends not only on the values of $u$ near $x$ but also $\forall x \in \mathbb{R}^N$. Hence, the fractional Laplacian $(-\Delta)^{s}$ is said to be a nonlocal operator. Now, we are going to define the function spaces. \\

The fractional Sobolev space $H^s(\Omega)$ is defined by:
$$H^s(\R^N):=\left\{ u\in L^2(\R^N):   [u]_{H^s(\R^N)} <\infty \right\},$$ 
and 
$$H_0^s(\Omega):=\left\{ u\in H^s(\R^N):   u(x)=0 \text{ in } \R^N \setminus \Omega \right\},$$

where the seminorm is also known as the Gagliardo-Nirenberg semi-norm of $u$ defined as
$$[u]_{H^s(\Omega)}^2:= \int_{\R^N}\int_{\R^N} \frac{(u(x)-u(y))^2}{|x-y|^{N+2s}}\;\mathrm{d}x\;\mathrm{d}y,$$ 

where $H^{s}(\Omega)$ is the Hilbert space along with the norm
$$\|u\|_{H^s(\Omega)}^2:=\|u\|_{L^2(\Omega)}^2+ [u]_{H^s(\Omega)}^2.$$
 
Also the closure of $C_0^\infty(\Omega)$ in $H^s(\R^N)$ with respect to the norm $\|.\|_{H^s(\Omega)}$ is defined as $H^s_0(\Omega)$, i.e., $H^s_0(\Omega):=\overline{{C_0^\infty(\Omega)}}^{H^s(\R^N)}$.\\

Also for $s>\frac{1}{2}$, we have the important identity(see  \cite{fiscella2015density})
 $${H^s_0(\overline{\Omega})}={H^s_0(\Omega)}. $$

The dual space of $H^s_0(\Omega)$ is known as $H^{-s}(\Omega)$ which is also a Hilbert space along with the norm $\|.\|_{H^{-s}(\Omega)}$.

For $0<t\leq s <1, ~H^s(\Omega)\hookrightarrow H^t(\Omega)$ (see \cite{di2012hitchhikers}).\\

Hence, the following Equation follows from the Gelfand triplet and \cite{di2012hitchhikers} follows. 
The following identity: 
$$H^s_0(\Omega) \subset H^s(\R^N) \subset L^2(\R^N) \subset H^{-s}(\R^N) \subset H^{-s}(\Omega).$$

In this article, we are going to study the integral fractional Laplacian approach. In particular, we have the following identities (see \cite{di2012hitchhikers}):

$$\lim_{s\rightarrow{1^{-}}}(-\Delta)^s u = -\Delta u ~~~~ \lim_{s\rightarrow{0^{+}}}(-\Delta)^s u=u.$$ 

PDE-constrained optimization is a new, rapidly growing area. The study of optimization techniques for solving such problems is quite interesting as well as challenging both numerically and analytically. Interested readers can go through(see \cite{troltzsch2010optimal}, \cite{lions1971optimal}, \cite{antil2017spectral}) for exposure to PDE-constrained optimization problems.                                                   
Let  $s\in (0,1)$ be a real number and $\Omega \subset \mathbb{R}^N$ be a bounded and regular domain. Let us consider the fractional Poisson equation: 
                                                             
\begin{equation} \label{eq:1.1}
    \begin{cases}
        (-\Delta)^s u_s = f_s,  ~~~~  x \in \Omega, \\
                    u_s = 0,    ~~~~  x \in \mathbb{R}^{N}
    \end{cases}
\end{equation}
where the integral definition of fractional Laplacian is the following:
$$(-\Delta)^s u_s(x):= C_{N,s} \text{P.V.} \int_{\mathbb{R}^N} \frac{u_s(x)-u_s(y)}{|x-y|^{N+2s}}\; \mathrm{d}y,$$
where $C_{N,s}$ be the normalising constant is defined in the above $\ref{e1}$ .

\begin{definition} [see \cite{di2012hitchhikers}]Let $f_s\in H^{-s}(\Omega)$. A function $u_s \in H^s(\Omega) $ is said to be the weak solution of the problem \ref{eq:1.1} if                                                           
$$ C_{N,s}\int_{\mathbb{R}^N} \int_{\mathbb{R}^N} \frac{(u_s(x)-u_s(y))(v(x)-v(y))}{|x-y|^{N+2s}} \; \mathrm{d} x  \mathrm{d} y = \int_{\Omega} f_sv \; \mathrm{d} x \quad \forall v \in H^s_0(\Omega).$$
\end{definition}
The left-hand side of the above term induces the semi-norm $[.]_{H^s(\Omega)}$, which will be discussed in the next section.

Consider the map $J_s:H^s_0(\Omega) \to \mathbb{R}$ defined by $u\mapsto \|u\|^2_{H^s(\Omega)}$ be functional which is Fr/'echt differentiable in 
$u \in H^s_0(\Omega)$ for any $v\in H^s_0(\Omega)$.\\

We can write $$\lim_{t\to 0}\frac{ [u+tv]^2_{H^s(\Omega)}-[u]^2_{H^s(\Omega)}}{t}= 2\int_{\R^N} \int_{\R^N} \frac{(u(x)-u(y))(v(x)-v(y))}{|x-y|^{N+2s}}\;\mathrm{d}x\;\mathrm{d}y,$$
also
$$(u_1,u_2)_{H^s(\R^N)}= 2 C_{N,s}^{-1}((-\Delta)^\frac{s}{2}u_1,(-\Delta)^\frac{s}{2}u_2)_{L^2(\R^N)} \\ 
 =2 C_{N,s}^{-1}(u_1,-\Delta^s u_2)_{L^2(\R^N)}.$$

Hence, we can write (see \cite{di2012hitchhikers}), 
$$\label{seminorm} [u]^2_{H^s(\R^N)}= C_{N,s}^{-1}\|(-\Delta)^\frac{s}{2} u\|^2_{L^2(\R^N)},$$
where 
$$(u_1,u_2)_{H^s(\R^N)}:= \int_{R^N}\int_{\R^N}\frac{(u_1(x)-u_1(y))(u_2(x)-u_2(y))}{|x-y|^{N+2s}}\;\mathrm{d}x\;\mathrm{d}y.$$

\begin{lemma} [see \cite{cea}]
A bounded sequence  $\left\{g_n\right\}_{n=1}^\infty $ in the Hilbert space has a weakly convergent subsequence $\left\{g_{nk}\right\}_{k=1}^\infty $.
\end{lemma}

\begin{definition}[\textbf{$\Gamma-$convergence} \cite{braides2002oxford}] 
$F$ be the Gamma limit of $F_n$  or in other words $ F$ is the $\Gamma$-$\lim_{n\to \infty} F_n$.
The sequential limit of $F_n$ is the function $F$ if  the following conditions are satisfied: 
\\

(i) For every $x\in X$ and for every sequence $\left\{x_n\right\}$ converging to $x$ in $X$ into $\R \cup \left\{-\infty,+\infty\right\} $ then $$\liminf_{n\to \infty}F_n(x_n)\geq F(x).$$
\\
(ii) For every $x\in X$,there exists a sequence $\left\{x_n\right\}$ converging to $x\in X$ such that 
                $$\lim_{n\to \infty}F_n(x_n)=F(x).$$
 \end{definition}                           
The following lemma is crucial to prove our main result.
\begin{lemma}[see \cite{braides2002oxford}] 
\label{lemmaB}
Let $x_n^*$ be the minimizer of $F_n$ for each $n$. Let $F$ be the $\Gamma-$limit of $F_n$ and $x_n^* \to x^*$ for some $x^*$. Let $E$ be  the set containing all those points where F's value is finite, then $x^*$ be the minimizer of $F$ and 
$$\lim_{n\to \infty}F_n(x_n^*) = F(x^*).$$
\end{lemma}

\section{Main Problem}
Let $ s\in (0,1)$ and $ \Omega \subset \mathbb{R}^N $ be the open and bounded set with Lipschitz boundary $\partial \Omega$. Consider  $\mu >0$ be the regular parameter, control $f_s\in L^2(\Omega)$, Our proposed PDE-constrained optimized  problem is the following : 
\begin{equation}\label{eq:3.1}
\min J_s(u_s,f_s):=\frac{1}{2}\Big([u_s]^2_{H^s(\Omega)}+\mu\|f_s\|^2_{L^2(\Omega)} \Big)  
\end{equation}
subject to (the {\textit{fractional Poisson equation})
\begin{equation}\label{eq:3.2}
 \begin{cases}
(-\Delta)^s u_s = f_s ,  & x\in \Omega, \\
u_s=0 ,     & x\in  \Omega^c, 
    \end{cases} \quad \quad  
    \end{equation}
and the control constraints
\begin{equation}\label{eq:3.3}
 a \leq \|f_s\|_{L^2(\Omega)} \leq b.   
\end{equation}
\\
Consider  $U_{ad}$  is the set of  admissible controls such that 
$$U_{ad}= \left\{ f_s \in L^2(\Omega) \; | \quad a \leq \|f_s\|_{L^2(\Omega)} \leq b \; \; \text{a.e.} \;\; x \in \Omega \right\} $$ for  $a,b \in \R$.

As the state $u_s$ depends on control $f_s$, then $u_s$ can also be written as $u_s= u_s(f_s)$ and so from now we can use notation  $J_s(f_s)$ instead of   $J_s(u_s(f_s),f_s)$.
In the above problem  \ref{eq:3.1},\ref{eq:3.2},\ref{eq:3.3} we are going to find a optimal control $\bar{f}_s \in U_{ad} \subset  L^2(\Omega)$ in a such a way that the corresponding solution of $\bar{u}_s$ together with $\bar{f}_s$ satisfies the minimization of the cost function, i.e.,\\
\begin{equation*}
    J_s(\bar{f}_s):=\min_{f_s \in U_{ad}} J_s(f_s),
\end{equation*}
and corresponding optimal control of classical Poisson equation with homogeneous boundary conditions is the following 
\begin{equation}\label{eq:3.4}
    \min J(u,f)=\frac{1}{2}\Big(\| \nabla u \|^2_{L^2(\Omega)}+\mu \|f\|^2_{L^2(\Omega)} \Big) 
\end{equation}
subject to (the$~Poisson ~equation $)

\begin{equation}
     \begin{cases}
(-\Delta)u = f ,  & x\in \Omega, \\
u=0,     & x\in  \partial \Omega,
\end{cases} 
\end{equation}
and the control constraints 
\begin{equation}
   a \leq \|f\|_{L^2(\Omega)} \leq b.  
\end{equation}
Our main objective is to find the optimal control  $\bar{f}$ such that 
\begin{equation*}
    J(\bar{f})=\min_{f \in U_{ad}} J(f).
\end{equation*}
\begin{proposition} [see \cite{biccari2017poisson}]\label{prop} Let $\mathcal{F}_s=\left\{f_s\right\}_{0<s<1} \subset H^{-s}(\Omega)$ be the sequence satisfying $\|f_s\|_{H^{-s}(\Omega)} $ uniformly bounded with respect to $s$ and $f_s \rightharpoonup f$ weakly in $H^{-1}(\Omega)$ as $s\to 1^-$, then $u_s \to u $ strongly in $H_0^{1-\delta}(\Omega)$ for some $C>0$ and $0<\delta \leq 1$.
\end{proposition}
Here, we will extend this proposition \ref{eq:1.1} to optimal control of fractional PDE in our main result.

The first proposition represents about the Poincar\'e inequality in fractional Sobolev space and the second refers to the minimizer of a function defined on a suitable space.
 
\begin{proposition}[Poincar\'e inequality, \cite{bonder2020optimal}] Let $s\in(0,1),\Omega \subset \R^N$ be an open and bounded set then we have 
$$\|u\|^2_2\leq C(N,\Omega,s)[u]^2_{H^s(\Omega)}.$$
where some constant $C(N,\Omega,s)$ depending upon $N,\Omega \text{and} s$.
\end{proposition}

\begin{proposition} Let the admissible control space $U_{ad}$ be weakly closed, bounded subset of $L^2(\Omega)$ with $J_s:U_{ad}\to \R$ is weakly lower semi-continuous. Then $J_s$ has minimizer in $U_{ad}$.\\
\end{proposition}

The existence and uniqueness of optimal control via strict convex and lower semi-continuous is discussed in the following proposition.  
\begin{proposition} $J_s$ is defined as above satisfying the equations (\ref{eq:3.1}), (\ref{eq:3.2}), (\ref{eq:3.3}), then there exist a unique minimizer of $J_s$.
\end{proposition}
\begin{proof} It is well known that the Gagliardo-Nirenberg seminorm is convex, and the quadratic function is strictly convex. 
So strict convexity of quadratic function $[u_s]^2_s$ and $\|f_s\|^2_2$, we can obtain strict convexity of $J_s$. By Poincare inequality and Cauchy-Schwarz inequality, the fractional Poisson equation (\ref{eq:3.2}) implies 
$$[u_s]^2_{H^s(\Omega)}=\int_{\Omega}f_s u_s \;\mathrm{d}x \leq C\|f_s\|_2[u_s]_{H^s(\Omega)}$$

$$\implies [u_s]_{H^s(\Omega)} \leq C\|f_s\|_2.$$
   
Which implies the control of the state operator 
$$S:L^2(\Omega)\to H^s_0(\Omega).$$ defined by $f_s \mapsto u_s$ is strongly continuous.
Thus, the map $f_s \to [u_s]^2_{H^s(\Omega)}$ is strongly continuous from $L^2(\Omega)$ into $\R$. Strict convexity and strong continuous together implies that $J_s$ is weakly lower semi-continuous. By the above propositions, we derived that $J_s$  has a minimizer in $U_{ad}$. Strict convexity implies the uniqueness of the minimizer on a convex set. So $J_s$ has a unique minimizer in $U_{ad}$.
\end{proof}

Now, we are ready to describe the main results of our study.

\section{The behavior of the optimal control problem as $s \to 1^{-} $}
\label{sec:4}

\begin{proposition}[\cite{di2012hitchhikers}] 
\label{propositionB}
Let us consider a fixed $u\in L^2(\R^N)$ with $u\equiv0$ in $\R^N \setminus \Omega$. Then,
\begin{equation*}
\lim_{s\to 1^-} [u]^2_{H^s(\R^N)}=
\begin{cases}
C_{N,s}\|\nabla u\|^2_2    & u \in H^1(\R^N)\\
\infty  & u \notin H^1(\R^N)  
\end{cases}
\end{equation*} and in the sense of distribution $\lim_{s\to 1^-} (-\Delta)^su=-\Delta u$.
\end{proposition}

\begin{proposition}[see \cite{bourgain2001another}]
Given a sequence $s\to 1^-$ and $\left\{u_s\right\}_s$ be such that 
\begin{equation*}
\sup_{s}\|u_s\|_2 < \infty \text{ and }     \sup_{s} [u_s]_{H^s(\Omega)} <  \infty
\end{equation*}
then there exists a function $u\in H^1(\R^N)$ such that (upto a subsequence),
\begin{equation*}
u_s \to u\text{ strongly in } L_{loc}^2(\R^N) \text{  and   } C_{N,s} \|\nabla u \|^2_2 \leq \liminf_{s\to 1^-} [u_s]_{H^s(\Omega)}^2.
\end{equation*}
\end{proposition}

\begin{proposition}[see \cite{bonder2020optimal}] 
\label{propositionD}
Given $s\in(0,1)$,let $f_s \in L^2(\Omega)$  be such that $f_s $ weakly convergent to $f$ in $L^2(\Omega)$ and $u_s\in H^s_0(\Omega)$ and $u \in H^1_0(\Omega)$ be the solution to 
\begin{equation}
\begin{cases}
(-\Delta)^s u_s = f_s ,  & x\in \Omega, \\
u_s=0,     & x\in \R^N \setminus \Omega,
\end{cases} 
\end{equation}
and
\begin{equation}
\begin{cases}
(-\Delta) u = f ,  & x\in \Omega, \\
u=0,     & x\in \partial \Omega,
\end{cases} 
\end{equation}
respectively. Then $u_s \to u$ strongly in $L^2(\Omega)$. Moreover as $s\to 1^-$
\begin{equation}
[u_s]_{H^s(\Omega)}^2 \to C_{N,s} \|\nabla u\|^2_2,
\end{equation} where $C_{N,s}$ be the normalising constant defined as before .
\end{proposition}
 
We are going to see the limiting behavior of 
\begin{equation}
\min J_s(f_s)=\frac{1}{2}\Big([u_s]^2_{H^s(\Omega)}+\mu\|f_s\|^2_{L^2(\Omega)} \Big). 
\end{equation}
where $f_s$ is the solution of the fractional Poisson equation, let $u_s^*$ be the optimal control, i.e., be the unique minimizer of $J_s$ in $U_{ad}$. Also, by hypothesis, the $f_s^*$ is the bounded in $L^2(\Omega)$. A subsequence weakly convergent exists, say $f^*$ in $L^2(\Omega)$. Now we extend the cost functional given in \ref{eq:3.1} to all $L^2(\Omega)$ as follows:
 
\begin{equation}
F_s(f_s) = 
\begin{cases}
J_s(f_s),   & f_s \in U_{ad},\\
\infty ,    &f_s \in L^2(\Omega)\setminus U_{ad}.
\end{cases}
\end{equation}

Similarly, we can extend the local optimal control problem cost function (\ref{eq:3.4})

\begin{equation}
F(f) = 
\begin{cases}
J(f), &  f\in U_{ad}, \\
\infty, &  f \in L^2(\Omega)\setminus U_{ad}.
\end{cases}
\end{equation}  

where $C_{N,s}$ is the normalizing constant here. We need to modify the definition of the following Galiardo seminorm in fractional Sobolev space in the rest of the part of the paper, 
$$[u_s]^2_{H^s(\Omega)} := C_{N,s} \int_{\R^N}\int_{\R^N} \frac{(u(x)-u(y))^2}{|x-y|^{N+2s}}.$$ for convenience notation.

Let us define $u_k:=u_{f_k} ,J_{s_k}:=J_{k}$  and $f_k\in H^{s_k}(\Omega)$.
Now, we are ready to state and prove our main result.

\begin{theorem}[Main Theorem] \label{thm}
Under the above notation, let $u_k\in H^{s_k}_0(\Omega)$ and $u\in H^1_0(\Omega)$,\\
optimal control of nonlocal Poisson equation with zero exterior boundary conditions 
\begin{equation}
\min J_k(f_k)=\frac{1}{2}\Big([u_k]^2_{H^{s_k}(\Omega)}+\mu\|f_k\|^2_{L^2(\Omega)} \Big) 
\end{equation}
subject to (the fractional Poisson equation)
\begin{equation}
\begin{cases}
(-\Delta)^s u_k = f_k ,  & x\in \Omega, \\
u_k=0,     & x\in \R^N \setminus \Omega,
\end{cases} 
\end{equation}
and the control constraints,
\begin{equation}
 f_{s_k} \in U_{ad}\subset L^2(\Omega),
\end{equation}
and  corresponding  optimal control of the local Poisson equation  with homogeneous boundary conditions is the following 
\begin{equation}
    \min J(u,f)=\frac{1}{2}\Big(\| \nabla u \|^2_{L^2(\Omega)}+\mu \|f\|^2_{L^2(\Omega)} \Big) 
\end{equation}
subject to (the Poisson equation)

\begin{equation}
     \begin{cases}
(-\Delta)u = f ,  & x\in \Omega, \\
u=0,     & x\in  \Omega^c,
\end{cases} 
\end{equation}
and the control constraints 
\begin{equation}
   f \in U_{ad}\subset L^2(\Omega).
\end{equation}
Also consider $f_k^*$ be the minimizer of $J_k$ that is   $J_k(f_k^*)=\min_{f_k \in U_{ad}}J_k(f_k)$ and 
  $f^*$ be the minimizer of $J$ that is  $J(f^*)=\min_{f \in U_{ad}}J(f)$.
We have the following (up to the subsequence), for $s_k\in(0,1),s_k\to 1^-$  as $k\to \infty$,
\begin{enumerate}
\item  $f_k^*\rightharpoonup f^*$ in $L^2(\Omega)$.
\item  $u_k^*\to u^*$ in $L^2(\Omega)$.
\item  $J_k(f_k^*)\to J(f^*) $ in $\R$.
\end{enumerate}

\end{theorem} 

\begin{proof}
Let $\left\{f_k^*\right\}$ be the set of optimal control for each $s_k$. Also $\left\{f_k^*\right\}\subset U_{ad}$. The admissible control set $U_{ad}$ is the bounded subset of the Hilbert space $L^2(\Omega)$. Hence, for each sequence $\left\{f_k^*\right\}$ in $U_{ad}$ admits a weakly convergent subsequence in (still the same notation) 
\begin{equation*}
f_k^* \rightharpoonup \bar{f} \text{ as }s\to 1^- \text{ for some } \bar{f}\in U_{ad}.
\end{equation*}\\
\textbf{claim:} $\bar{f}=f^*$ \\
It is enough to prove that 
"$F$  be the Gamma limit of $F_k$" or in other words, $$ F=\Gamma-\lim_{k\to \infty} F_k,$$ 

and from Lemma \ref{lemmaB} our results are immediate.\\
Let given any sequence $f_s\rightharpoonup  f$ in $L^2(\Omega)$, we will derive the
 
$$F(f)\leq \liminf_{k\to \infty}F_k(f_k).$$

As $f_k\rightharpoonup f$ in $L^2(\Omega)$, then by Proposition \ref{propositionD}, we have $u_k \to u $ strongly in $L^2(\Omega)$ as $k\to \infty $.

\begin{equation*}
\begin{split}
\liminf_{k\to \infty}F_k(f_k) &= \lim_{k\to \infty}\Big(\frac{1}{2}[u_k]_{H^{s_k}(\Omega)}^2 + \frac{\mu}{2}\|f_k\|^2 \Big)\\
&=\lim_{k\to \infty} \frac{1}{2}[u_k]_{H^{s_k}(\Omega)}^2 + \liminf_{k\to \infty}\frac{\mu}{2}\|f_k\|^2 \\
&\geq \lim_{k\to \infty} \frac{1}{2}[u_k]_{H^{s_k}(\Omega)}^2 +  \frac{\mu}{2}\|f\|^2 \\
& =  \frac{1}{2}\|\nabla u\|^2_2 +  \frac{\mu}{2}\|f\|^2 = F(f).
\end{split}
\end{equation*}

Hence $$F(f)\leq \liminf_{k\to \infty}F_k(f_k).$$

It is remaining to show that there exists, $ \forall f\in L^2(\Omega), \exists f_k \to f$, such that 
$$\lim_{k\to \infty}F_k(f_k)=F(f).$$
By choosing $f_k=f,\forall k $ then we have $\lim_{k\to \infty}f_k=f$ trivially. For this sequence, we have the following

\begin{equation*}
\begin{split}
\lim_{k\to \infty} F_k(f_k) &= \lim_{k\to \infty }\Big(\frac{1}{2}[u_k]_{H^{s_k}(\Omega)}^2 + \frac{\mu}{2}\|f_k\|^2 \Big)\\
& = \lim_{k\to \infty }\frac{1}{2}[u_k]_{H^{s_k}(\Omega)}^2+ \lim_{k\to \infty}  \frac{\mu}{2}\|f_k\|^2 \\
& =  \frac{1}{2}\|\nabla u\|^2_2 +  \frac{\mu}{2}\|f\|^2 = F(f).
\end{split}
\end{equation*}
Hence by definition of $\Gamma-$convergence we have $F=\Gamma-\lim_{k\to \infty} F_k$ in $L^2(\Omega)$. If $f_k^*$ is the optimal control of $J_k$ and $f^*$ is the optimal control of $J$. By Lemma \ref{lemmaB} we have that  $J(f^*)=\lim J_k(f_k^*)$.
Also, because of  Proposition \ref{propositionB}  together with 

$$J(f^*)=\lim J_k(f_k^*),$$ implies

 $$\|f_k^*\|^2 \to \|f^*\|^2.$$  
Hence, we have 
$$f_k^*\to f^* \text{ strongly in } L^2(\Omega). $$
This is stronger than what we wanted in the claim. 
Therefore, we established $f^*=\bar{f}$.
                                                      
The remaining part of the claim $u_k^*\to u^*$ strongly in $L^2(\Omega)$ directly follows from Proposition \ref{propositionD}.
\end{proof}
\section{Conclusion} \label{sec:5}
We have concluded that the optimal state solution of the controlled fractional Poisson equation is in the vicinity of the optimal state solution of the controlled classical Poisson equation.
\bibliographystyle{plain}
\bibliography{KNBVRRM120325}
\end{document}